\documentclass[11pt]{article}                                                                                                                                                                    

\usepackage{amsthm,amsmath,amssymb,graphics,graphicx,hyperref,epstopdf}
\usepackage{tikz}
\usetikzlibrary{shapes}
\usepackage{multicol}

\title{Balanced Non-Transitive Dice II: Tournaments}
\author{Alex Schaefer\footnote{Binghamton University; aschaef3@binghamton.edu}}
\newtheorem{theorem}{Theorem}[section]

\newtheorem{definition}[theorem]{Definition}

\newtheorem{question}[theorem]{Question}
\newtheorem{example}[theorem]{Example}

\newcommand{\ssection}[1]{%
  \section[#1]{\centering\normalfont\scshape #1}}

\begin{document} 

\maketitle

\begin{abstract}
We further study sets of labeled dice in which the relation ``is a better die than'' is non-transitive.  Focusing on sets with an additional symmetry we call ``balance,'' we prove that sets of $n$ such $m$-sided dice exist for all $n,m \geq 3$. We then show how to construct a set of $n$ dice such that the relation behaves according to the direction of the arrows of \emph{any} tournament (complete directed graph) on $n$ vertices.
\end{abstract}

\ssection{Introduction}

Consider the following game: choose a die in Figure \ref{dice}, and then I choose a different die (based on your choice).  We roll our dice, and the player whose die shows a higher number wins.  \\

In the long run, I will have an advantage in this game:  Whichever die you choose, I will choose the one immediately to its left (and I will choose die C if you choose die A).  In any case, the probability of my die beating yours is $19/36 > 1/2$.  \\

\begin{figure}[htp]
\centering
\includegraphics[height = 1.05in]{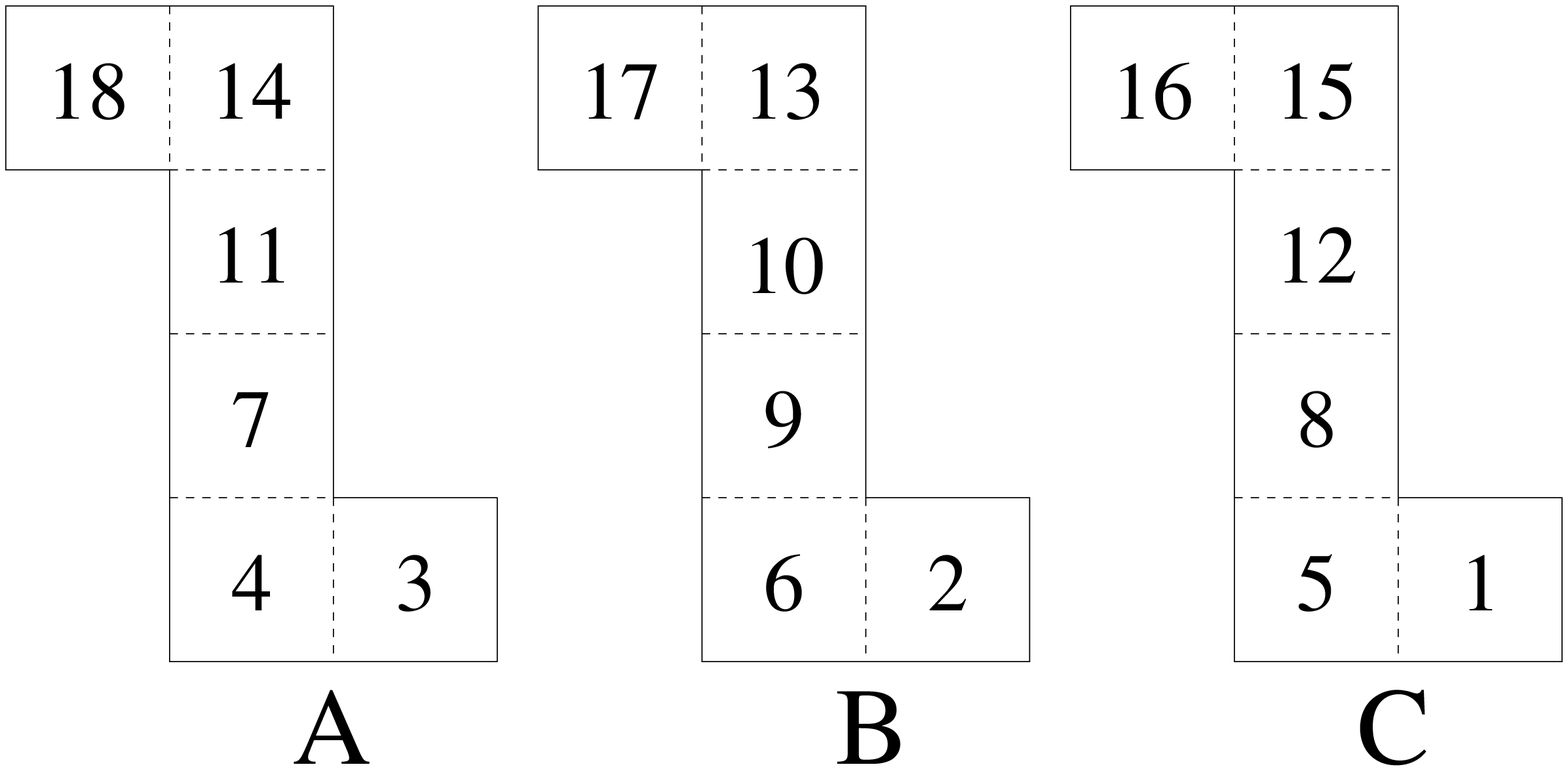}
\caption{A set of balanced non-transitive 6-sided dice.}\label{dice}
\end{figure}

This is a case of the phenomenon of \emph{non-transitive} dice, first introduced by Martin Gardner in \cite{MGntD}, and further explored in \cite{MGnt}, \cite{ntD1}, and \cite{FFF}.  \\

Fix integers $n, m \geq 3$.  For our purposes, a set of $n$ $m$-sided \emph{dice} is a collection of pairwise-disjoint sets $A_{1}, A_{2},\dots, A_{n}$ with $|A_{i}| = m$ and $\cup A_{i} = [n\cdot m]$ (here and throughout, $[k] = \{1, 2, \dots, k\}$).  We think of die $A_{i}$ as being labeled with the elements of $A_{i}$.  Each die is fair, in that the probability of rolling any one of its numbers is $1/m$.  We also write $P(A\succ B)$ for the probability that, upon rolling both $A$ and $B$, the number rolled on $A$ exceeds that on $B$, and $A\succ B$ if this probability exceeds $\frac{1}{2}$.

\begin{definition}
A set of dice is \emph{non-transitive} if $A_{i} \succ A_{i+1}$ for all $i$ (and $A_{n} \succ A_{1}$).  That is, the relation ``is a better die than'' is non-transitive.
\end{definition}

In this paper we (mostly) examine non-transitive sets of dice, but we introduce a new property as well.

\begin{definition}
A set of dice is \emph{balanced} if $P(A_{i} \succ A_{i+1}) = P(A_{j} \succ A_{j+1}) = P(A_{n} \succ A_{1})$ for all $i$ and $j$. This value is called the \emph{victorious probability} of the set.  
\end{definition}

Note that the set of dice in Figure \ref{dice} is balanced, as $P(A \succ B) = P(B \succ C) = P(C \succ A) = 19/36$.  \\

\begin{definition}
A \emph{graph} is an ordered pair $(V,E)$ where $V$ is a set of \emph{vertices} and $E$ is a set of unordered pairs of vertices called \emph{edges}. A \emph{directed graph} is a graph where the edges are ordered pairs.
\end{definition}

We can then also think of a set of dice as the set of vertices of a graph, and having $P(A\succ B)>\frac{1}{2}$ may correspond to a directed edge from $A$ to $B$.

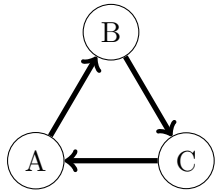
\begin{figure}[h!]
\begin{center}
\begin{tikzpicture}[scale=.5]
\node[draw, shape=circle] (1) at (-2,0) {A};
\node[draw, shape=circle] (2) at (2,0) {C};
\node[draw, shape=circle] (3) at (0,3.42) {B};

\draw[->, line width=2pt] (1)--(3);
\draw[->, line width=2pt] (3)--(2);
\draw[->, line width=2pt] (2)--(1);
\end{tikzpicture}
\end{center}
\caption{The dice of figure \ref{dice} as an orientation of $C_{3}$.}\label{C3}
\end{figure}

In \cite{bntd1}, Schaefer and Schweig showed that non-transitive balanced sets of $n$ $m$-sided dice exist for $n=3,4$ and all $m \geq 3$.

\begin{definition}
A directed graph $G$ is \emph{realizble} by a set of dice $D$ if there is a one-to-one map $f:D\to V(G)$ such that $P(A\succ B)>\frac{1}{2}$ implies that $(f(A),f(B))$ is a directed edge of $G$.
\end{definition}

The main results of \cite{bntd1} could then be restated as follows.

\begin{theorem}\label{C3andC4realizability}
Directed $3$- and $4$-cycles are realizable by sets of balanced dice with any number $m\geq3$ of sides.
\end{theorem}

Our first goal here is to generalize this statement so that the cycle may also be of any length. Then, we will generalize again from directed $C_{n}$ (cycles) to directed $K_{n}$ (complete graphs).

\ssection{Realizing cycles as dice}

The main goal in this section is to prove the following.

\begin{theorem}\label{Cnrealizability}
For any $n,m \geq 3$, there exists a non-transitive set of $n$ balanced $m$-sided dice.
\end{theorem}

An example to illustrate our procedure will be useful.

\begin{example}

We start with a set of balanced non-transitive dice, and would like to add another one.

$\begin{array}{cccc}A: & 9 & 5 & 1 \\B: & 8 & 4 & 3 \\C: & 7 & 6 & 2\end{array} \mapsto \begin{array}{cccc}A: & \hat{9} & \hat{5} & \hat{1} \\B: & \hat{8} & \hat{4} & \hat{3} \\C: & \hat{7} & \hat{6} & \hat{2}\\ D: & ? & ? & ?\end{array}$

We require that $C\succ D$ and $D\succ A$. But We already have $C\succ A$, and so these three dice are totally ordered: $C>D>A$. As such, we can move outside of $\mathbb{N}$ by copying $C$ to $D$ and lowering all values by, say, $1/10$.

$\begin{array}{cccc}A: & 9 & 5 & 1 \\B: & 8 & 4 & 3 \\C: & 7 & 6 & 2 \\D: & 6.9 & 5.9 & 1.9\end{array}$

This new set of dice does indeed have $C\succ D\succ A$, and so we are done if we only seek non-transitivity. However, the original set was balanced, and this will likely not be (we have $P(C\succ D)=\binom{m+1}{2}/m^{2}>P(D\succ A)$). So, if we count the number of ``victories'' of our original set (the numerator of our probability), we can raise values on $D$ by $1/10$ instead of lowering them to lower the number of victories of $C$ over $D$. We can alter this number by any amount we desire, from $1$ to $\binom{m}{2}$ (by raising every value on $D$), and so can match the desired victorious probability. The last step is to return to $\mathbb{N}$ by relabeling linearly.

$\begin{array}{cccc}A: & 9 & 5 & 1 \\B: & 8 & 4 & 3 \\C: & 7 & 6 & 2 \\D: & 6.9 & 5.9 & 1.9\end{array} \mapsto \begin{array}{cccc}A: & 9 & 5 & 1 \\B: & 8 & 4 & 3 \\C: & 7 & 6 & 2 \\D: & 6.9 & 5.9 & \textbf{2.1} \end{array} \mapsto \begin{array}{cccc}A: & 12 & 6 & 1 \\B: & 11 & 5 & 4 \\C: & 10 & 8 & 2 \\D: & 9 & 7 & 3\end{array}.$
\end{example}

This procedure is general.

\begin{proof}[Proof of Theorem \ref{Cnrealizability}]
We proceed by induction. Our base case, with three dice ($n=3$), is done (for arbitrary number of sides $m\geq3$, see \cite{bntd1}). So assume we have a set of $k$ balanced non-transitive $m$-sided dice, $A_{1},\dots,A_{k}$. Create a new die $A_{k+1}$ whose entries are those of $A_{k}$, each shifted down by some $\epsilon<1$. The set of dice $A_{1},\dots,A_{k+1}$ could be relabeled linearly from $[(k+1)m]$, which would complete the proof if the condition of balance were omitted from the theorem (we could then also omit it from the proof). But, by shifting the entries of $A_{k+1}$ up by $\epsilon$ rather than down, we alter $P(A_{k}\succ A_{k+1})$ while keeping $P(A_{k+1}\succ A_{1})$ the same (the victorious probability we started with). This recovers the condition of balance.
\end{proof}

\ssection{Tournaments}

A \emph{tournament} is a directed complete graph. Two vertices $x,y$ in a directed graph are \emph{strongly connected} if there is a directed path from $x$ to $y$ and also one from $y$ to $x$. Under this equivalence relation, the vertices of a directed graph are sorted into \emph{strongly connected components} (or \emph{strong components}). A \emph{strongly connected directed graph} is one with only one strong component. We know from Moon \cite{moon} that a tournament is strong if and only if it contains a directed cycle of every length. He also shows, in particular, a tournament is strong if and only if it contains a directed Hamilton cycle.

Given a directed graph, we may form a new directed graph from it by contracting each connected component down to a single vertex. The result will likely have parallel edges, but all edges between any two vertices point in the same direction; delete all but one of each parallel edge group. The result, called the \emph{condensation}, is always acyclic. For a strong directed graph, the condensation is a single vertex.

Because each vertex in the condensation of a directed graph contains a directed cycle on all its vertices of the original directed graph, we can give a set of non-transitive dice (one for each vertex of the condensation) that realizes the cycle. The question, then, is about any edges between vertices not adjacent (with an edge in either direction) in the cycle. Namely, can we choose or manipulate our dice to obey these edges as well? We will answer this question by constructing dice that realize any tournament.

If the tournament is not strong, the condition of balance will be impossible. However, strong tournaments can be realized by balanced dice, and so we can create a set of balanced dice for each vertex in the condensation (strong component), and then shift the labels to obey the total order given by the condensation. The problem of realizing tournaments then reduces to the problem of strong tournaments, which we will use to our advantage.

Given a strong tournament, locate a directed Hamilton cycle as a subgraph. This cycle alone can be realized by balanced non-transitive dice by Theorem \ref{Cnrealizability}. We will then augment our dice to account for the other edges.

Again, an example will be helpful.

\begin{example}
Start with a directed 5-cycle, and a set of $5$ balanced non-transitive dice with $3$ sides (constructed by repeating the procedure of the previous example).

\begin{center}
\begin{multicols}{3}
\begin{tikzpicture}[scale=.4]
\node[draw, shape=circle] (1) at (90:4) {A};
\node[draw, shape=circle] (2) at (18:4) {B};
\node[draw, shape=circle] (3) at (306:4) {C};
\node[draw, shape=circle] (4) at (234:4) {D};
\node[draw, shape=circle] (5) at (162:4) {E};

\draw[->, line width=2pt] (1)--(2);
\draw[->, line width=2pt] (2)--(3);
\draw[->, line width=2pt] (3)--(4);
\draw[->, line width=2pt] (4)--(5);
\draw[->, line width=2pt] (5)--(1);
\end{tikzpicture}
\columnbreak

\columnbreak
$\begin{array}{cccc} \\ A: & 15 & 7 & 1 \\ B: & 14 & 6 & 5 \\ C: & 13 & 10 & 2 \\ D: & 12 & 9 & 3 \\ E: & 11 & 8 & 4 \end{array}$
\end{multicols}
\end{center}

For every edge we add, we will need to add sides to our dice: one above and one below, for a total of $2\left(\binom{n}{2}-n\right)=n^{2}-3n$ extra entries. Because half of them are below, we shift all our labels up by $\left(\binom{n}{2}-n\right)=\frac{n^{2}-3n}{2}$.

\begin{center}
\begin{multicols}{3}
\begin{tikzpicture}[scale=.4]
\node[draw, shape=circle] (1) at (90:4) {A};
\node[draw, shape=circle] (2) at (18:4) {B};
\node[draw, shape=circle] (3) at (306:4) {C};
\node[draw, shape=circle] (4) at (234:4) {D};
\node[draw, shape=circle] (5) at (162:4) {E};

\draw[->, line width=2pt] (1)--(2);
\draw[->, line width=2pt] (2)--(3);
\draw[->, line width=2pt] (3)--(4);
\draw[->, line width=2pt] (4)--(5);
\draw[->, line width=2pt] (5)--(1);
\end{tikzpicture}
\columnbreak

\columnbreak
$\begin{array}{cccc} \\ A: & 25 & 17 & 11 \\ B: & 24 & 16 & 15 \\ C: & 23 & 20 & 12 \\ D: & 22 & 19 & 13 \\ E: & 21 & 18 & 14 \end{array}$
\end{multicols}
\end{center}

We will add the missing edges (which can be done in any order) by choosing the two numbers above and the two numbers below our existing labels. Count the number of victories that existed to begin with on the missing edge. For three-sided non-transitive dice, it will be either $4$ or $5$. If the die we want to be victorious had $5$, it gets the smaller number of the two below (it doesn't need one more). Otherwise it gets the larger. The other die gets the opposite. Of the two numbers above, the larger goes on the die we want to be victorious (the smaller on the other). So to add the edge $(A,C)$, $A$ will get the larger of $\{9,10\}$ (as currently $P(A\succ C)=\frac{4}{9}$), and it also gets the larger of $\{26,27\}$ (so that $A$ will beat $C$).

\begin{center}
\begin{multicols}{3}
\begin{tikzpicture}[scale=.4]
\node[draw, shape=circle] (1) at (90:4) {A};
\node[draw, shape=circle] (2) at (18:4) {B};
\node[draw, shape=circle] (3) at (306:4) {C};
\node[draw, shape=circle] (4) at (234:4) {D};
\node[draw, shape=circle] (5) at (162:4) {E};

\draw[->, line width=2pt] (1)--(2);
\draw[->, line width=2pt] (2)--(3);
\draw[->, line width=2pt] (3)--(4);
\draw[->, line width=2pt] (4)--(5);
\draw[->, line width=2pt] (5)--(1);
\draw[->, line width=2pt] (1)--(3);
\end{tikzpicture}
\columnbreak

\columnbreak
$\begin{array}{cc|ccc|c}A: & 27 & 25 & 17 & 11 & 10\\ B: & & 24 & 16 & 15 & \\ C: & 26 & 23 & 20 & 12 & 9 \\ D: & & 22 & 19 & 13 & \\ E: & & 21 & 18 & 14 & \end{array}$
\end{multicols}
\end{center}

Now add, say, $(B,D)$. That means $B$ gets $29$ and $D$ gets $28$. In the original, $P(B\succ D)=\frac{5}{9}$, so $B$ gets $7$ and $D$ gets $8$.

\begin{center}
\begin{multicols}{3}
\begin{tikzpicture}[scale=.4]
\node[draw, shape=circle] (1) at (90:4) {A};
\node[draw, shape=circle] (2) at (18:4) {B};
\node[draw, shape=circle] (3) at (306:4) {C};
\node[draw, shape=circle] (4) at (234:4) {D};
\node[draw, shape=circle] (5) at (162:4) {E};

\draw[->, line width=2pt] (1)--(2);
\draw[->, line width=2pt] (2)--(3);
\draw[->, line width=2pt] (3)--(4);
\draw[->, line width=2pt] (4)--(5);
\draw[->, line width=2pt] (5)--(1);
\draw[->, line width=2pt] (1)--(3);
\draw[->, line width=2pt] (2)--(4);
\end{tikzpicture}
\columnbreak

\columnbreak
$\begin{array}{cc|ccc|c}A: & 27 & 25 & 17 & 11 & 10 \\ B: & 29 & 24 & 16 & 15 & 7 \\ C: & 26 & 23 & 20 & 12 & 9 \\ D: & 28 & 22 & 19 & 13 & 8 \\ E: & & 21 & 18 & 14 & \end{array}$
\end{multicols}
\end{center}

Note that $(A,B)$ (and all others) remain correct! When we add $(B,D)$, a value larger than any yet on $A$ appears on $B$, but so does a value smaller than any yet on $A$, so the net change on the number of victories of $A$ over $B$ is zero.

We proceed in this fashion. The order the edges were added here was: $(B,E)$, $(C,E)$, $(A,D)$. This gives the set of dice below. If edges were added in a different order, a different set of dice, also with the features of this one, would be produced.

\begin{center}
\begin{multicols}{3}
\begin{tikzpicture}[scale=.4]
\node[draw, shape=circle] (1) at (90:4) {A};
\node[draw, shape=circle] (2) at (18:4) {B};
\node[draw, shape=circle] (3) at (306:4) {C};
\node[draw, shape=circle] (4) at (234:4) {D};
\node[draw, shape=circle] (5) at (162:4) {E};

\draw[->, line width=2pt] (1)--(2);
\draw[->, line width=2pt] (2)--(3);
\draw[->, line width=2pt] (3)--(4);
\draw[->, line width=2pt] (4)--(5);
\draw[->, line width=2pt] (5)--(1);
\draw[->, line width=2pt] (1)--(3);
\draw[->, line width=2pt] (2)--(4);
\draw[->, line width=2pt] (2)--(5);
\draw[->, line width=2pt] (3)--(5);
\draw[->, line width=2pt] (1)--(4);
\end{tikzpicture}
\columnbreak

\columnbreak
$\begin{array}{ccc|ccc|cc}A: & 35 & 27 & 25 & 17 & 11 & 10 & 2 \\ B: & 31 & 29 & 24 & 16 & 15 & 7 & 5 \\ C: & 33 & 26 & 23 & 20 & 12 & 9 & 3 \\ D: & 34 & 28 & 22 & 19 & 13 & 8 & 1 \\ E: & 32 & 30 & 21 & 18 & 14 & 6 & 4 \end{array}$
\end{multicols}
\end{center}

\end{example}

\begin{theorem}\label{strongtourney}
Let $G$ be a strong tournament. There is a set of balanced non-transitive dice realizing $G$.
\end{theorem}

\begin{proof}
Let $G$ be a strong tournament on $n\geq3$ vertices; we construct a set of dice which will have $3+2(n-3)=2n-3$ sides (to make the condition of balance easy to recover). Because $G$ is strong, it contains a directed Hamilton cycle. By Theorem \ref{Cnrealizability}, there is a set $\{A_{1},\dots,A_{n}\}$ of balanced non-transitive dice realizing this subgraph of $G$ with $3$ sides. The only possible victorious probability of such a set of dice is $\frac{5}{9}$. Further, the number of victories of \emph{any} die over any other (i.e. those not adjacent in the cycle) is either $4$ or $5$. We will add sides to our dice to account for the edges other than those in the cycle.\\

Add $\frac{n^{2}-3n}{2}$ (an integer) to all the labels for the $A_{i}$, which will now be labeled by $\{\frac{n^{2}-3n}{2},\frac{n^{2}-3n}{2}+1,\dots,3n+\frac{n^{2}-3n}{2}\}$. Choose a directed edge $(A_{i},A_{j})$ other than one in the cycle, and assume (without loss) $A_{i}\succ A_{j}$. Calculate $P(A_{i}\succ A_{j})$ as it stands in the original set, which is either $\frac{4}{9}$ or $\frac{5}{9}$. Append the label $3n+\frac{n^{2}-3n}{2}+2$ to $A_{i}$ and $3n+\frac{n^{2}-3n}{2}+1$ to $A_{j}$. Append the smaller two labels as follows:

\begin{tabular}{c|c} $A_{i}$ gets $3n+\frac{n^{2}-3n}{2}-2$, $A_{j}$ gets $3n+\frac{n^{2}-3n}{2}-1$ & if $P(A_{i}\succ A_{j})=\frac{5}{9}$, \\ the opposite & If $P(A_{i}\succ A_{j})=\frac{4}{9}$. \end{tabular}\\

The placement of the larger number on $A_{i}$ will add $5$ additional victories, and the placement of \emph{a} smaller number on $A_{j}$ will add $3$ victories to $A_{i}$ (one for each original element of $A_{i}$ and either $0$ or $1$ more based on what $P(A_{i}\succ A_{j})$ was to begin with. The new victorious probability is $\frac{13}{25}$.\\

It is clear that the repetition of this process allows the new set of dice to obey the edges outside the cycle. We also remark that we do not negatively impact the edges in the cycle: if, at the $j^{\text{th}}$ stage we append two numbers to $A_{i}$, and in step $k>j$ we append two numbers to $A_{i+1}$, $A_{i+1}$ gains $k$ extra victories over $A_{i}$ from the large label, but $k$ extra losses from the small label.\\

We can iterate through all the ``chords'' of the Hamilton cycle (in any order), and we finish with a set of dice labeled by $[3n+2\left(\binom{n}{2}-n\right)]=[n^{2}]$ which is balanced, non-transitive, and obeys all arrows in the tournament.
\end{proof}

\ssection{Further Questions}

If a directed graph is acyclic, it is trivially realizable even by one sided dice (acyclic graphs correspond to total orderings). By this and previous observations, along with Theorem \ref{strongtourney}, we have the following, which puts everything together nicely.

\begin{theorem}\label{realizability}
Let $G$ be a directed graph. There is a set of dice realizing $G$. Moreover, the dice may be made balanced if and only if $G$ is a subgraph of a strong tournament.
\end{theorem}

The condition on balance suggests the following.

\begin{definition}
A directed graph is \emph{strongly connectable} if the missing edges can be added and directed in a way such that the resulting tournament is strong.
\end{definition}

\begin{question}
Is there a necessary and sufficient condition for strong connectability?
\end{question}

This question was answered (in this form) by Joyce, Schaefer, West, and Zaslavsky, who showed in \cite{JSWZ} that the obvious necessary condition on $G$ of containing no complete directed cut (a complete collection of edges between a partition of the vertices into two nonempty sets, all pointing one way) is also sufficient.\\

The method of proof of Theorem \ref{strongtourney} suggests:

\begin{question}
What is the minimum number of sides required for a set of $n$ dice to realize a strong tournament? Any tournament?
\end{question}

A natural generalization is to the realm of weighted directed graphs.

\begin{question}
Let $G$ be a weighted directed graph. Can $G$ be realized by a set of dice such that the probability of one die beating another depends (in some way) on the weight of the edge between the corresponding vertices?
\end{question}

\end{document}